\documentclass[12pt,a4paper]{amsart}
\usepackage{amsfonts}
\usepackage{amsthm}
\usepackage{amsmath}
\usepackage{amscd}
\usepackage[latin2]{inputenc}
\usepackage{t1enc}
\usepackage[mathscr]{eucal}
\usepackage{indentfirst}
\usepackage{graphicx}
\usepackage{graphics}
\usepackage{pict2e}
\usepackage{epic}
\usepackage{url}
\numberwithin{equation}{section}
\usepackage[margin=2.9cm]{geometry}

\theoremstyle{plain}
\newtheorem{Th}{Theorem}[section]
\newtheorem{Lemma}[Th]{Lemma}
\newtheorem{Cor}[Th]{Corollary}
\newtheorem{Prop}[Th]{Proposition}

 \theoremstyle{definition}
\newtheorem{Def}[Th]{Definition}
\newtheorem{Conj}[Th]{Conjecture}
\newtheorem{Rem}[Th]{Remark}
\newtheorem{?}[Th]{Problem}
\newtheorem{Ex}[Th]{Example}

\newcommand{\G}{\mathbb{G}}
\newcommand{\prm}{\mathrm{pm}}

\begin{document}

\title[New proof of Schrijver's and Gurvits's theorems]{Lower matching conjecture, and a new proof of Schrijver's and Gurvits's theorems}

\author[P. Csikv\'ari]{P\'{e}ter Csikv\'{a}ri}

\address{Massachusetts Institute of Technology \\ Department of Mathematics \\
Cambridge MA 02139 \&  E\"{o}tv\"{o}s Lor\'{a}nd University \\ Department of Computer 
Science \\ H-1117 Budapest
\\ P\'{a}zm\'{a}ny P\'{e}ter s\'{e}t\'{a}ny 1/C \\ Hungary} 

\email{peter.csikvari@gmail.com}

\thanks{The author is partially supported by National Science Foundation under grant no. DMS-1500219,
the Hungarian National Research, Development and Innovation Office, NKFIH grant K109684 and NN114614, a Slovenian--Hungarian grant, and by the MTA R\'enyi ''Lend\"ulet'' Groups and Graphs Research Group, and by the ERC Consolidator Grant 648017.}



\begin{abstract} Friedland's Lower Matching Conjecture asserts that if $G$ is a $d$--regular bipartite graph on $v(G)=2n$ vertices, and $m_k(G)$ denotes the number of matchings of size $k$, then
$$m_k(G)\geq {n \choose k}^2\left(\frac{d-p}{d}\right)^{n(d-p)}(dp)^{np},$$
where $p=\frac{k}{n}$. When $p=1$, this conjecture reduces to a theorem of Schrijver which says that a $d$--regular bipartite graph on $v(G)=2n$ vertices has at least 
$$\left(\frac{(d-1)^{d-1}}{d^{d-2}}\right)^n$$
perfect matchings. L. Gurvits proved an asymptotic version of the Lower Matching Conjecture, namely he proved that
$$\frac{\ln m_k(G)}{v(G)}\geq \frac{1}{2}\left(p\ln \left(\frac{d}{p}\right)+(d-p)\ln \left(1-\frac{p}{d}\right)-2(1-p)\ln (1-p)\right)+o_{v(G)}(1).$$

In this paper, we prove the Lower Matching Conjecture. In fact, we will prove a slightly stronger statement which gives an extra  $c_p\sqrt{n}$ factor compared to the conjecture if $p$ is separated away from $0$ and $1$, and is tight up to a constant factor if $p$ is separated away from $1$. We will also give a new proof of Gurvits's and Schrijver's theorems, and we  extend these theorems  to $(a,b)$--biregular bipartite graphs.  
\end{abstract}

\maketitle

\section{Introduction} Throughout this paper we use standard terminology, but the second paragraph of Section 2 might help the Reader 
in case of a concept undefined in the Introduction.
\medskip

One of the best known theorem concerning the number of perfect matchings of a $d$--regular graph is due to A. Schrijver and M. Voorhoeve.

\begin{Th}[A. Schrijver \cite{sch1} for general $d$, M. Voorhoeve \cite{vor} for $d=3$] \label{Schrijver} Let $G$ be a $d$--regular bipartite graph on $2n$ vertices and let $\prm(G)$ denote the number of perfect matchings of $G$. Then
$$\prm(G)\geq \left(\frac{(d-1)^{d-1}}{d^{d-2}}\right)^n.$$
\end{Th}

There are two different proofs of Theorem~\ref{Schrijver}. The original one is due to A. Schrijver \cite{sch1}, and another proof using stable polynomials is due to L. Gurvits \cite{gur}, for a beautiful account to this proof see \cite{L-S}.
In this paper we will give a third proof of this theorem which is essentially different from the previous ones.

In \cite{FKM}, S. Friedland, E. Krop and K. Markstr\"om, conjectured a possible generalization of this theorem which extends Schrijver's theorem to any size of matchings. This conjecture became known as Friedland's Lower Matching Conjecture:

\begin{Conj}[Friedland's Lower Matching Conjecture \cite{FKM}] \label{LMC} Let $G$ be a $d$--regular bipartite graph on $v(G)=2n$ vertices, and let $m_k(G)$ denote the number of matchings of size $k$, then
$$m_k(G)\geq {n \choose k}^2\left(\frac{d-p}{d}\right)^{n(d-p)}(dp)^{np},$$
where $p=\frac{k}{n}$.
\end{Conj}

They also proposed an asymptotic version of this conjecture which was later proved by L. Gurvits in \cite{gur2}.

\begin{Th}[L. Gurvits \cite{gur2}] \label{Gurvits} Let $G$ be a $d$--regular bipartite graph on $v(G)=2n$ vertices,  and let $m_k(G)$ denote the number of matchings of size $k$, then
$$\frac{\ln m_k(G)}{v(G)}\geq \frac{1}{2}\left(p\ln \left(\frac{d}{p}\right)+(d-p)\ln \left(1-\frac{p}{d}\right)-2(1-p)\ln (1-p)\right)+o_{v(G)}(1),$$
where $p=\frac{k}{n}$.
\end{Th}

When $p=1$ this result almost reduces to Schrijver's theorem, but Gurvits used this special case to establish the general case. More precisely, Gurvits used the following result of Schrijver: let $A=(a_{ij})$ be a doubly stochastic matrix, and $\tilde{A}=(\tilde{a}_{ij})$, where $\tilde{a}_{ij}=a_{ij}(1-a_{ij})$, then the permanent of $\tilde{A}$ satisfies the inequality
$$\mathrm{Per}(\tilde{A})\geq \prod_{i,j}(1-a_{ij}).$$
Note that Gurvits \cite{gur2} proved an effective version of Theorem~\ref{Gurvits}, but for our purposes any $o_{v(G)}(1)$ term would suffice, as we will "vanish" it. More details on Gurvits's results can be found at Remark~\ref{rem: Gurvits}. 

It is worth introducing some notation for the function appearing in Theorem~\ref{Gurvits}, and with some foresight we introduce another function with parameters $a,b$ which will be important for us when we study $(a,b)$--biregular graphs.

\begin{Def} Let $0\leq q\leq 1$ and
$$H(q)=-(q\ln (q)+(1-q)\ln (1-q))$$
with the usual convention that $H(0)=H(1)=0$.
Furthermore, for a positive integer $d$ and $0\leq p\leq 1$ let
$$\G_d(p)=\frac{1}{2}\left(p\ln \left(\frac{d}{p}\right)+(d-p)\ln \left(1-\frac{p}{d}\right)-2(1-p)\ln (1-p)\right),$$
and for positive integers $a$ and $b$, let
$$\G_{a,b}(p)=\frac{a}{a+b}H\left(\frac{a+b}{2a}p\right)+\frac{b}{a+b}H\left(\frac{a+b}{2b}p\right)+\frac{1}{2}p\ln (ab)-\frac{ab}{a+b}H\left(\frac{a+b}{2ab}p\right),$$
where $0\leq p\leq \min(\frac{2a}{a+b},\frac{2b}{a+b})$.
\end{Def}

Note that one can rewrite $\G_{a,b}(p)$ as follows:
$$\G_{a,b}(p)=\frac{1}{2}\left(p \cdot\ln \left( \frac{2ab}{(a+b)p}\right)+\left(\frac{2ab}{a+b}-p\right)\cdot \ln \left(1-\frac{a+b}{2ab}p\right)-\right.$$
$$-\left. \left(\frac{2a}{a+b}-p\right)\cdot \ln \left(1-\frac{a+b}{2a}p\right)-
\left(\frac{2b}{a+b}-p\right)\cdot \ln \left(1-\frac{a+b}{2b}p\right)\right).$$
From this form it is clear that for $a=b=d$, we have $\G_d(p)=\G_{a,b}(p)$. Later it will turn out that $\G_d(p)$ is the so-called entropy function of the infinite 
$d$--regular tree $\mathbb{T}_d$, and $\G_{a,b}(p)$ is the entropy function of the infinite $(a,b)$--biregular tree $\mathbb{T}_{a,b}$.
\bigskip

To show the connection between Conjecture~\ref{LMC} and Theorem~\ref{Gurvits}, let us introduce one more parameter. Let $p=\frac{k}{n}$, and let $p_{\mu}$ be the probability that  a random variable with distribution Binomial$(n,p)$ takes its mean value $\mu=k$.
In other words,
$$p_{\mu}={n \choose k}p^k(1-p)^{n-k}.$$
With this new notation the function appearing in Conjecture~\ref{LMC} is
$${n \choose k}^2\left(\frac{d-p}{d}\right)^{n(d-p)}(dp)^{np}=p_{\mu}^2\exp (2n\G_d(p)).$$
Hence Conjecture~\ref{LMC} claims that
$$m_k(G)\geq p_{\mu}^2\exp (2n\G_d(p)).$$
It turns out that a slightly stronger statement is true.

\begin{Th} \label{main} Let $G$ be a $d$--regular bipartite graph on $v(G)=2n$ vertices,  and let $m_k(G)$ denote the number of matchings of size $k$. Furthermore, let $p=\frac{k}{n}$, and $p_{\mu}$ be the probability that  a random variable with distribution Binomial$(n,p)$ takes its mean value $\mu=k$. Then
$$m_k(G)\geq p_{\mu}\exp (2n\G_d(p)).$$
In particular, Conjecture~\ref{LMC} holds true. Furthermore, for every $0\leq k<n$ there exists a $d$--regular bipartite graph $G$ on $2n$ vertices such that
$$m_k(G)\leq \sqrt{\frac{1-p/d}{1-p}}\cdot p_{\mu}\exp(2n \G_d(p)).$$
\end{Th}

Note that $p_{\mu}\approx \frac{1}{\sqrt{2\pi p(1-p)n}}$, this means that if $p$ is separated away from $0$ and $1$, then we can obtain an extra $c_p\sqrt{n}$ factor compared to Conjecture~\ref{LMC}. Also note that $p_{\mu}\geq \frac{1}{n+1}\geq \frac{1}{2n}$ always holds true. This inequality might be easier to handle in some cases.
\bigskip

We will practically show that Theorem~\ref{Gurvits} implies Conjecture~\ref{LMC}. The idea of the proof of Theorem~\ref{main} is to convert Gurvits's theorem to a statement on analytical functions arising from statistical mechanics. Then tools from analysis and probability theory together with a simple observation will enable us to replace the term $o_{v(G)}(1)$ in Gurvits's theorem with an effective one which is slightly better than the corresponding term in the original Gurvits's theorem (see Remark~\ref{rem: Gurvits}).
\bigskip

We offer one more theorem for $d$--regular bipartite graphs.

\begin{Th} \label{direct} Let $G$ be a $d$--regular bipartite graph on $v(G)=2n$ vertices,  and let $m_k(G)$ denote the number of matchings of size $k$. Let $0\leq p\leq 1$, then
$$\sum_{k=0}^nm_k(G)\left(\frac{p}{d}\left(1-\frac{p}{d}\right)\right)^{k}(1-p)^{2(n-k)}\geq \left(1-\frac{p}{d}\right)^{nd}.$$
\end{Th}

When $p=1$, Theorem~\ref{direct} immediately gives back Theorem~\ref{Schrijver}. Indeed, when $p=1$ only the term $m_n(G)\left(\frac{1}{d}\left(1-\frac{1}{d}\right)\right)^n$
does not vanish on the left hand side, because of the term $(1-p)^{2(n-k)}$, and we get that
$$m_n(G)\left(\frac{1}{d}\left(1-\frac{1}{d}\right)\right)^n\geq \left(1-\frac{1}{d}\right)^{nd}$$
which is equivalent with 
$$m_n(G)\geq \left(\frac{(d-1)^{d-1}}{d^{d-2}}\right)^n.$$
\bigskip

As we mentioned Theorem~\ref{Gurvits} implies Theorem~\ref{main}, but the main goal of this paper is to give a new proof of Gurvits's and Schrijver's theorems with a novel method. This method will be used to prove Theorem~\ref{direct} too.  This new proof shows that the extremal graph is in some sense the $d$--regular infinite tree. Indeed, we will show that the function on the right hand side of Theorem~\ref{Gurvits} is nothing else than the so-called entropy function of the $d$--regular infinite tree;  the entropy functions of finite and infinite graphs  will be introduced in Section 2. It means that for a deeper understanding of these theorems, one needs to step out from the universe of finite graphs. We will do it by the recently developed theory of Benjamini--Schramm convergence of bounded degree graphs. This new technique also enables us to extend  these theorems to $(a,b)$--biregular bipartite graphs.

\begin{Th} \label{biregular} Let $G=(A,B,E)$ be an $(a,b)$--biregular bipartite graph on $v(G)$ vertices such that every vertex in $A$ has degree $a$, and every vertex in $B$ has degree $b$. Assume that $a\geq b$, i. e., $|A|\leq |B|$.  Let $m_k(G)$ denote the number of matchings of size $k$, and $p=\frac{2k}{v(G)}$. Furthermore, let $q=\frac{a+b}{2b}p$, and let $p_{\mu}$ be the probability that  a random variable with distribution Binomial$(|A|,q)$ takes its mean value $\mu=k$. Then
$$m_k(G)\geq p_{\mu}\exp(v(G)\cdot \G_{a,b}(p)).$$
\end{Th}

Note that if $k=|A|$, then $p_{\mu}=1$, and $p_{\mu}\geq \frac{1}{|A|+1}\geq \frac{1}{v(G)}$ holds true for any $p$.
\medskip

One can view Theorem~\ref{biregular} and the other results as extremal graph theoretic problems where one seeks for the extremal value of a certain graph parameter $p(G)$ in a given family $\mathcal{G}$ of graphs. In extremal graph theory it is a classical idea to try to find some graph transformation $\varphi$ such that $p(G)\leq p(\varphi(G))$ (or $p(G)\geq p(\varphi(G))$), and $\varphi(G)\in \mathcal{G}$ for every $G\in \mathcal{G}$. Then we apply this transformation as long as we can, and when we stop then we know that the extremal graph must be in a special subfamily of $\mathcal{G}$, where the optimization problem can be solved easily. See for instance the proof of Tur\'an's theorem using Zykov's symmetrization \cite{zyk}. In our case, the transformation $\varphi$ will be simply any $2$-lift of the graph (see Definition~\ref{def: 2-lift}). The new ingredient in our proof is that the sequence of graphs obtained by applying repeatedly the $2$-lifts will not stabilize, but instead converge to the infinite biregular tree. In fact, most of our work is related to the graph convergence part, and not the graph transformation part.
\bigskip

\textbf{This paper is organized as follows.} In the next section we introduce all the necessary tools including the density function $p(G,t)$, the entropy function $\lambda_G(p)$, the Benjamini--Schramm convergence, and the computation of the entropy function of the infinite biregular tree. In this section we also give various results on the number of matchings of random (bi)regular graphs which shows the tightness of our results. In particular, we prove the second half of Theorem~\ref{main} here. In Section 3 we show that Gurvits's theorem is equivalent with certain (effective) statement on the entropy function. In Section 4 we give the new proof of Schrijver's and Gurvits's theorem together with the main part of the proof of Theorem~\ref{biregular}. In Section 5 we deduce Theorem~\ref{main} from the new version of Gurvits's theorem, we prove Theorem~\ref{direct}, and we also finish the proof of  Theorem~\ref{biregular}.
\bigskip

\textbf{Some advice how to read this paper.} This paper is occasionally a bit technical, especially Section~\ref{entropy-function}. In order to make it easier to read this paper we roughly summarize this paper and give a road map for the first reading. Assuming that the Reader is mainly interested in the proof of Theorem~\ref{main} we first give an idea how the proof works. 

Assume that $p(G)$ is some graph parameter related to matchings and it is normalized in such a way that we can compare graphs on different sizes, in particular it makes sense to compare two $d$--regular graphs. For instance
$$p(G)=\frac{\ln \textrm{pm}(G)}{v(G)}$$
is such a graph parameter. We will prove that for a bipartite $d$--regular graph we have
$$p(G)\geq p(\mathbb{T}_d),$$
where $p(\mathbb{T}_d)$ apriori does not make sense, but can be defined as a limit $\lim_{i\to \infty}p(H_i)$, where $H_i$ is a sequence of graphs "converging locally" to $\mathbb{T}_d$. The plan is the following: we define a sequence of graphs $G_i$ such that $G=G_0$ and
$$p(G)=p(G_0)\geq p(G_1)\geq p(G_2)\geq p(G_3)\geq \dots $$
and
$$\lim_{i \to \infty}p(G_i)=p(\mathbb{T}_d).$$
This clearly gives that
$$p(G)\geq p(\mathbb{T}_d).$$

A technical difficulty arises from the fact it is not really convenient to work with the parameter
$$q(G)=\frac{\ln m_k(G)}{v(G)}.$$
Instead we use the entropy function $\lambda_G(p)$ which is strongly related to the parameter $q(G)$, but is more amenable to any analysis. So 
$\lambda_G(p)$ will play the role of $p(G)$. Of course, we need some tools to transfer our knowledge from $\lambda_G(p)$ to $q(G)$, but it is again just a technical problem.

So by keeping in mind our very simple plan and our techincal difficulty we suggest the following road map for the first reading:
(1) first read the alternative definition of $\lambda_G(p)$, Remark~\ref{alternative entropy}, this is a definition which is easy to understand, then take a quick look at its properties, Proposition~\ref{asymp} without reading its proof, (2) read the definition of Benjamini--Schramm convergence, Definition~\ref{BS-convergence} and Example~\ref{large-girth}, (3) jump to Section~\ref{new proof}, read it only till the proof of Theorem~\ref{Schrijver} and \ref{Gurvits} (4) finally read Section~\ref{finish}. We believe that only reading this core of the paper will give a good impression of the content of this paper and the novel method applied in this paper. 

Let us mention that if the Reader is familiar with Gurvits's result, Theorem~\ref{Gurvits} and only wants to know how one can derive Theorem~\ref{main} from it then after step (1) in the above plan one can jump immediately to Section~\ref{effective form} and then Section~\ref{finish}.

\section{Preliminaries and basic notions} \label{entropy-function}

This section is mostly reproduced from the paper \cite{ACH}. We could have simply cited this paper, but for the sake of the convenience of the Reader, we also included the proofs. 
\bigskip

Throughout the paper, $G$ denotes a finite graph with vertex set $V(G)$ and
edge set $E(G)$. The number of vertices is denoted by $v(G)$. The \emph{degree} of a vertex
is the number of its neighbors. A graph is called $d$\emph{--regular} if every
vertex has degree exactly $d$.  A \emph{cycle} $C$ is a 
sequence of vertices $v_1,v_2,\dots ,v_k$ such that $v_i\neq v_j$ if $i\neq j$ and $(v_i,v_{i+1})\in E(G)$ for $i=1,\dots ,k$, where $v_{k+1}=v_1$. The length of the cycle is $k$ in this case. A \emph{$k$--matching} is a set of edges $\{e_1,\dots ,e_k\}$ such that for any $i$ and $j$, the vertex set of $e_i$ and $e_j$ are disjoint, in other words, $e_1,\dots ,e_k$ cover $2k$ vertices together. A \emph{perfect matching} is a matching which covers every vertices. A graph is called \emph{bipartite} if the vertices can be colored with two colors such that all edges connect two vertices of different colors. The standard notation for bipartite graph is $G=(A,B,E)$, where $A$ and $B$ denote the vertex sets corresponding to the two color classes.
\bigskip

Let $G=(V,E)$ be a finite graph on $v(G)$ vertices, and $m_k(G)$ denotes the number of $k$-matchings ($m_0(G)=1$). Let $t$ be a non-negative real number; in statistical mechanics it is called the activity. Let
$$M(G,t)=\sum_{k=0}^{\lfloor v(G)/2\rfloor}m_k(G)t^k,$$
and 
$$\mu(G,x)=\sum_{k=0}^{\lfloor v(G)/2\rfloor}(-1)^km_k(G)x^{v(G)-2k}.$$
We call $M(G,t)$ the matching generating function\footnote{In statistical mechanics, it is called the partition function of the monomer-dimer model.} , $\mu(G,x)$ the matching polynomial \cite{hei,god3,GG}. Clearly, they encode the same information.
Let
$$p(G,t)=\frac{2t\cdot \frac{d}{dt}M(G,t)}{v(G)\cdot  M(G,t)},$$
and 
$$F(G,t)=\frac{\ln M(G,t)}{v(G)}-\frac{1}{2}p(G,t) \ln(t).$$
We will call $p(G,t)$ the density function. Note that there is a natural interpretation of $p(G,t)$. Assume that we choose a random matching $M$ with probability proportional to $t^{|M|}$. Then the expected number of vertices covered by a random matching is $p(G,t)\cdot v(G)$. 

\noindent Let 
$$p^*(G)=\frac{2\nu(G)}{v(G)},$$
where $\nu(G)$ denotes the number of edges in the largest matching. If $G$ contains a perfect matching, then clearly $p^*=1$.
The function $p=p(G,t)$ is a strictly monotone increasing function which maps $[0,\infty)$ to $[0,p^*)$, where $p^*=p^*(G)$.  Therefore, its inverse function $t=t(G, p)$ maps  $[0,p^*)$ to $[0,\infty)$. (If $G$ is clear from the context, then we simply write $t(p)$ instead of $t(G,p)$.) Let
$$\lambda_G(p)=F(G,t(p))$$
if $p<p^*$, and $\lambda_G(p)=0$ if $p>p^*$.
Note that we have not defined $\lambda_G(p^*)$ yet. We define it as a limit:
$$\lambda_G(p^*)=\lim_{p\nearrow p^*}\lambda_G(p).$$
We will show that this limit exists, see part (c) of Proposition~\ref{asymp}. We will call $\lambda_G(p)$ the entropy function of the graph $G$.

The intuitive meaning of $\lambda_G(p)$ is the following. Assume that we want to count the number of matchings covering $p$ fraction of the vertices. 
Let us assume that it makes sense: $p=\frac{2k}{v(G)}$, and so we wish to count $m_k(G)$. Then
$$\lambda_G(p)\approx \frac{\ln m_k(G)}{v(G)}.$$
The more precise formulation of this statement will be given in Proposition~\ref{asymp}. 
\bigskip

\begin{Prop} \cite{ACH} \label{asymp} Let $G$ be a finite graph. \\ 
(a) Let $rG$ be the union of $r$ disjoint copies of $G$. Then
$$\lambda_G(p)=\lambda_{rG}(p).$$
(b) If $p<p^*$, then 
$$\frac{d}{dp}\lambda_G(p)=-\frac{1}{2}\ln t(p).$$
(c) The limit 
$$\lambda_G(p^*)=\lim_{p\nearrow p^*}\lambda_G(p)$$
exists. \\
(d) Let $k\leq \nu(G)$ and  $p=\frac{2k}{v(G)}$. Then
$$\left|\lambda_G(p)-\frac{\ln m_k(G)}{v(G)}\right|\leq \frac{\ln v(G)}{v(G)}.$$
(e) Let $k=\nu(G)$, then for $p^*=\frac{2k}{v(G)}$ we have
$$\lambda_G(p^*)=\frac{\ln m_k(G)}{v(G)}.$$
In particular, if $G$ contains a perfect matching then,
$$\lambda_G(1)=\frac{\ln \prm(G)}{v(G)}.$$
(f) If for some function $f(p)$ we have
$$\lambda_G(p)\geq f(p)+o_{v(G)}(1)$$
for all graphs $G$, then
$$\lambda_G(p)\geq f(p).$$
(g) If for some graphs $G_1$ and $G_2$ we have
$$\frac{\ln M(G_1,t)}{v(G_1)}\geq \frac{\ln M(G_2,t)}{v(G_2)}$$
for every $t\geq 0$, then
$$\lambda_{G_1}(p)\geq \lambda_{G_2}(p)$$
for every $0\leq p\leq 1$.

\end{Prop}

\begin{Rem} \label{alternative entropy} Part (a) and (d) of Proposition~\ref{asymp} together  suggest an alternative definition for the entropy function $\lambda_G(p)$ for $p<p^*$: let $(k_r)$ be a sequence of integers such that
$$\lim_{r\to \infty}\frac{2k_r}{rv(G)}=p$$
then
$$\lambda_G(p)=\lim_{r\to \infty}\frac{\ln m_{k_r}(rG)}{rv(G)}.$$
In general when we have an infinite graph $L$, say $\mathbb{Z}^d$, then it is a natural idea to consider a graph sequence $G_i$ converging to $L$ and and to take a sequence $(k_i)$ such that
$$\lim_{i\to \infty}\frac{2k_i}{v(G_i)}=p$$
then to consider
$$\lambda_L(p)=\lim_{i\to \infty}\frac{\ln m_{k_i}(G_i)}{v(G_i)}.$$
In this sense, this alternative definition is nothing else than to consider the "$G$-lattice", infinitely many disjoint copies of $G$ and approximate it with $G_i=iG$, the union of $i$ copies of $G$.

This alternative definition is much more natural, especially from a statistical physical point of view. On the other hand, this definition is hard to work with.

\end{Rem}

We will need some preparation to prove Proposition~\ref{asymp}. Among many others we will need the following fundamental theorem about the matching polynomial.

\begin{Th}[Heilmann and Lieb \cite{hei}] \label{Hei} The zeros of the matching polynomial
$\mu(G,x)$ are real, and if the largest degree $D$ of $G$ is greater than $1$,
then  all zeros lie in the interval $[-2\sqrt{D-1},2\sqrt{D-1}]$. 
\end{Th}

We will also use the following theorem of Darroch.

\begin{Lemma}[Darroch's rule \cite{dar}] Let $P(x)=\sum_{k=0}^na_kx^k$ be a polynomial with only positive coefficients and real zeros. If
$$k-\frac{1}{n-k+2}<\frac{P'(1)}{P(1)}<k+\frac{1}{k+2},$$
then $k$ is the unique number for which $a_k=\max(a_1,a_2,\dots, a_n)$. If, on the other hand,
$$k+\frac{1}{k+2}<\frac{P'(1)}{P(1)}<k+1-\frac{1}{n-k+1},$$
then either $a_k$ or $a_{k+1}$ is the maximal element of $a_1,a_2,\dots, a_n$.
\end{Lemma} 

\begin{proof}[Proof of Proposition~\ref{asymp}]
\noindent (a)  Note that 
$$M(rG,t)=M(G,t)^r$$
implying that $p(rG,t)=p(G,t)$ and $\lambda_{rG}(p)=\lambda_G(p)$.
\medskip

\noindent (b) Since
$$\lambda_G(p)=\frac{\ln M(G,t)}{v(G)}-\frac{1}{2}p\cdot \ln(t)$$
we have
$$\frac{d\lambda_G(p)}{dp}=\left(\frac{1}{v(G)}\cdot \frac{\frac{d}{dt}M(G,t)}{M(G,t)}\cdot \frac{dt}{dp}-\frac{1}{2}\left(\ln(t)+p\cdot \frac{1}{t}\cdot \frac{dt}{dp}\right)\right)=-\frac{1}{2}\ln(t),$$
since 
$$\frac{1}{v(G)}\cdot \frac{\frac{d}{dt}M(G,t)}{M(G,t)}=\frac{p}{2t}$$
by definition.
\medskip

\noindent (c) From $\frac{d}{dp}\lambda_G(p)=-\frac{1}{2}\ln t(p)$ we see that for $p>p(G,1)$, the function $\lambda_G(p)$ is monotone decreasing. (We can also see that $\lambda_G(p)$ is a concave-down function.) Hence 
$$\lim_{p\nearrow p^*}\lambda_G(p)=\inf_{p>p(G,1)}\lambda_G(p).$$
\medskip

\noindent (d) First, let us assume that $k<\nu(G)$. In case of $k=\nu(G)$, we will slightly modify our argument. Let $t=t(G,p)$ be the value for which $p=p(G,t)$. The polynomial
$$P(G,x)=M(G,tx)=\sum_{j=0}^nm_j(G)t^jx^j$$
considered as a polynomial in variable $x$, has only real zeros by Theorem~\ref{Hei}. Note that
$$k=\frac{pv(G)}{2}=\frac{P'(G,1)}{P(G,1)}.$$
Darroch's rule says that in this case $m_k(G)t^k$ is the unique maximal element of the coefficient sequence of $P(G,x)$. In particular
$$\frac{M(G,t)}{v(G)}\leq m_k(G)t^k\leq M(G,t).$$
Hence
$$\lambda_G(p)-\frac{\ln v(G)}{v(G)}\leq \frac{\ln m_k(G)}{v(G)}\leq \lambda_G(p).$$
Hence in case of $k<\nu(G)$, we are done.
\medskip

If $k=\nu(G)$, then let $p$ be arbitrary such that
$$k-\frac{1}{2}<\frac{pv(G)}{2}<k.$$
Again we can argue by Darroch's rule as before that
$$\lambda_G(p)-\frac{\ln v(G)}{v(G)}\leq \frac{\ln m_k(G)}{v(G)}\leq \lambda_G(p).$$
Since this is true for all $p$ sufficiently close to $p^*=\frac{2\nu(G)}{v(G)}$ and 
$$\lambda_G(p^*)=\lim_{p\nearrow p^*}\lambda_G(p),$$
we have
$$\left|\lambda_G(p^*)-\frac{\ln m_k(G)}{v(G)}\right|\leq \frac{\ln v(G)}{v(G)}$$
in this case too.
\medskip

\noindent (e) By part (a) we have $\lambda_{rG}(p)=\lambda_G(p)$. Note also that if $k=\nu(G)$, then $m_{rk}(rG)=m_k(G)^r$. Applying the bound from part (d) to the graph $rG$, we obtain that
$$\left|\lambda_G(p^*)-\frac{\ln m_k(G)}{v(G)}\right|\leq \frac{\ln v(rG)}{v(rG)}.$$
Since 
$$\frac{\ln v(rG)}{v(rG)}\to 0$$
as $r\to \infty$, we get that
$$\lambda_G(p^*)=\frac{\ln m_k(G)}{v(G)}.$$
\medskip

\noindent (f) This is again a trivial consequence of $\lambda_{rG}(p)=\lambda_G(p)$.
\medskip

\noindent (g) From the assumption it follows that for the relative sizes of the largest matchings, we have $\frac{\nu(G_1)}{v(G_1)}\geq \frac{\nu(G_2)}{v(G_2)}$, and if $\frac{\nu(G_1)}{v(G_1)}=\frac{\nu(G_2)}{v(G_2)}$, then
$$\frac{\ln m_{\nu(G_1)}(G_1)}{v(G_1)}\geq \frac{\ln m_{\nu(G_2)}(G_2)}{v(G_2)}.$$
So the statement is trivial if $p\geq \frac{2\nu(G_2)}{v(G_2)}$. So we can assume that $0\leq p< \frac{2\nu(G_2)}{v(G_2)}$.
Let us consider the minimum of the function $\lambda_{G_1}(p)-\lambda_{G_2}(p)$ on the interval $[0,\frac{2\nu(G_2)}{v(G_2)}]$. This minimum is either attained at some endpoints or inside the interval at a point where the derivative is $0$. Note that $\lambda_{G_1}(0)=\lambda_{G_1}(0)=0$. According to the part (b), the derivative of $\lambda_{G_1}(p)-\lambda_{G_2}(p)$ is 
$$-\frac{1}{2}\ln t(G_1,p)+\frac{1}{2}\ln t(G_2,p).$$
If it is $0$ at $p_0$ then $t(G_1,p_0)=t(G_2,p_0)$, but then with the notation $t=t(G_1,p_0)=t(G_2,p_0)$ we have
$$\lambda_{G_1}(p_0)=\frac{\ln M(G_1,t)}{v(G_1)}-\frac{1}{2}p_0\ln(t)\geq \frac{\ln M(G_2,t)}{v(G_2)}-\frac{1}{2}p_0\ln(t)=\lambda_{G_2}(p_0).$$
So at every possible minimum of $\lambda_{G_1}(p)-\lambda_{G_2}(p)$, the function is non-negative. So it is non-negative everywhere.
\end{proof}

\subsection{Benjamini--Schramm convergence and the entropy function} \label{measure}

In this part we extend the definition of the function $\lambda_G(p)$ for infinite lattices $L$, more precisely to  certain \textit{random rooted graphs}.

\begin{Def} \label{BS-convergence} Let $L$ be a probability distribution on (infinite) rooted graphs; we will call $L$ a \emph{random rooted graph}.
For a finite rooted graph $\alpha$ and a positive integer $r$, let $\mathbb{P}(L,\alpha,r)$ be the probability that the $r$-ball
centered at a random root vertex chosen from the distribution $L$ is isomorphic to $\alpha$.

For a finite graph $G$, a finite rooted graph $\alpha$ and a positive integer
$r$, let $\mathbb{P}(G,\alpha,r)$ be the probability that the $r$-ball
centered at a uniform random vertex of $G$ is isomorphic to $\alpha$. 

We say that a sequence $(G_i)$ of bounded degree graphs is \emph{Benjamini--Schramm
convergent} if for all finite rooted graphs $\alpha$ and $r>0$, the
probabilities $\mathbb{P}(G_i,\alpha,r)$ converge. Furthermore, we say that \emph{$(G_i)$ Benjamini-Schramm converges to $L$},
if for all positive integers $r$ and finite rooted graphs $\alpha$, $\mathbb{P}(G_i,\alpha,r)\rightarrow \mathbb{P}(L,\alpha,r)$.
\end{Def}

Note that Benjamini--Schramm convergence is also called local convergence. This refers to the fact that the finite graphs $G_i$ look locally more and more like the infinite graph $L$.

\begin{Ex} \label{grid}  Let us consider a sequence of boxes in $\mathbb{Z}^d$ where all sides converge to infinity. This will be a  Benjamini--Schramm convergent graph sequence since for every fixed $r$, we will pick a vertex which at least $r$-far from the boundary with probability converging to $1$. For all these vertices we will see the same neighborhood. This also shows that we can impose arbitrary boundary condition, for instance periodic boundary condition means that we consider the sequence of toroidal boxes.  Boxes and toroidal boxes will be Benjamini--Schramm convergent even together, and converges to a distribution which is a rooted $\mathbb{Z}^d$ with probability $1$. 
\end{Ex}

\begin{Ex} \label{large-girth} Recall that a  $k$-cycle of a graph $H$ is a sequence of vertices $v_1,\dots,v_k$ such that $v_i\neq v_j$ if $i\neq j$, and $(v_i,v_{i+1})\in E(H)$ for $1\leq i\leq k$, where $v_{k+1}=v_1$. For a graph $H$, let $g(H)$ be the length of the shortest cycle in $H$, this is called the \emph{girth} of the graph. 

Let $(G_i)$ be a sequence of $d$--regular graphs such that $g(G_i)\to \infty$, then $(G_i)$ Benjamini--Schramm converges to the rooted $d$--regular infinite tree $\mathbb{T}_d$. Note that if in a finite graph $G$ the shortest cycle has length at least $2k+1$ then the $k$-neighborhood of any vertex looks like the $k$-neighborhood of any vertex of an infinite $d$--regular tree.

Let $(G_i)$ be a sequence of $(a,b)$-biregular graphs such that $g(G_i)\to \infty$, then $(G_i)$ Benjamini--Schramm converges to the following distribution: with probability $\frac{a}{a+b}$ it is the infinite $(a,b)$--biregular tree $\mathbb{T}_{a,b}$ with root vertex of degree $b$, and with probability $\frac{b}{a+b}$ it is the infinite $(a,b)$--biregular tree $\mathbb{T}_{a,b}$ with root vertex of degree $a$. With slight abuse of notation we will denote this random rooted tree with $\mathbb{T}_{a,b}$ as well.
\end{Ex}

\begin{Rem} Not every random rooted graph can be obtained as a limit of Benjamini--Schramm convergent finite graphs. A necessary condition is that the random rooted graph has to be \emph{unimodular}, this is a certain reversibility property of a random rooted graph. On the other hand, it is not known that every \emph{unimodular random graph} can be obtained as a limit of Benjamini--Schramm convergent finite graphs. This is the famous Aldous--Lyons problem. The interested Reader can consult with the book \cite{LL}.
\end{Rem} 

The following theorem was known in many cases for thermodynamic limit in statistical mechanics. We also note that a modification of  the algorithm `CountMATCHINGS' in \cite{BGK+} yields an alternative proof of part (a) and (b) of this theorem.

\begin{Th} \cite{ACH} \label{entropy} Let $(G_i)$ be a Benjamini--Schramm convergent graph sequence of bounded degree graphs. Then the sequences of functions\\
(a) $$p(G_i,t),$$
(b) $$\frac{\ln M(G_i,t)}{v(G_i)}$$
converge to strictly monotone increasing continuous functions on the interval $[0,\infty)$. \\
Let $p_0$ be real number between $0$ and $1$ such that $p^*(G_i)\geq p_0$ for all $n$. Then\\
(c) $$t(G_i,p),$$
(d) $$\lambda_{G_i}(p)$$
are convergent for all $0\leq p<p_0$.
\end{Th}

\begin{Rem}  We mention that H. Nguyen and K. Onak \cite{ngu}, and independently G. Elek and G. Lippner \cite{ele} proved that for a Benjamini--Schramm convergent graph sequence $(G_i)$, the following limit exits:
$$\lim_{n\to \infty}p^*(G_i).$$
(Recall that $p^*(G_i)=\frac{2\nu(G_i)}{v(G_i)}$.)
\end{Rem} 

To prove Theorem~\ref{entropy}, we need some preparation. We essentially repeat the argument of the paper \cite{ACFK}.

The following theorem deals with the behavior of the matching polynomial in
Benjamini--Schramm convergent graph sequences. The matching measure was introduced in the paper \cite{ACFK}:

\begin{Def} The \emph{matching measure} of a finite graph is defined as
$$\rho_G=\frac{1}{v(G)}\sum_{z_i:\ \mu(G,z_i)=0}\delta(z_i),$$
where $\delta(s)$ is the Dirac-delta measure on $s$, and we take every $z_i$ into account with its multiplicity.
\end{Def}

In other words, the matching measure  is the probability measure of uniform
distribution on the zeros of $\mu(G,x)$.

\begin{Th}[\cite{ACFK,ACH}] \label{wc} Let $(G_i)$ be  a Benjamini--Schramm convergent bounded
degree graph sequence. Let $\rho_{G_i}$ be the matching measure of the graph $G_i$. Then the sequence $(\rho_{G_i})$ is
weakly convergent, i. e., there exists some measure $\rho_{L}$ such that for every bounded continuous function $f$, we have
$$\lim_{i\to \infty} \int f(z)\, d\rho_{G_i}(z)=\int f(z)\, d\rho_{L}(z).$$
\end{Th}

\begin{Rem} This theorem was first proved in \cite{ACFK}. The proof given there relied on a general result on graph polynomials given in  the paper \cite{csi}. To make this paper as self-contained as possible we sketch here a slightly different proof outlined in a remark of \cite{ACFK}.
\end{Rem}

\begin{proof} For a graph $G$ let $S(G)$ denote the multiset of zeros of the matching polynomial $\mu(G,x)$, and
$$p_k(G)=\sum_{\lambda\in S(G)}\lambda^k.$$
Then $p_k(G)/v(G)$ can be rewritten in terms of the measure $\rho_{G}$ as follows:
$$\frac{p_k(G)}{v(G)}=\int z^k\, d\rho_{G}(z).$$
It is known that $p_k(G)$ counts the number of
closed tree-like walks of length $k$ in the graph $G$: see chapter 6 of
\cite{god3}.  Without  going into the details of the description of `tree-like walks'; we only use the fact that
these are special type of walks, consequently we can count them
by knowing all $k$-balls centered at the vertices of the graph $G$.
Let $TW(\alpha)$ denote the number of closed tree-like walks starting at the root of $\alpha$, and let $\mathcal{N}_k$ denote the set of $k$-neighborhoods $\alpha$. The size of $\mathcal{N}_k$ is bounded by a function of $k$ and the largest degree of the graph $G$.
Furthermore, let $N_k(G,\alpha)$ denote the number of vertices of $G$ for which the $k$-neighborhood of the vertex is isomorphic to $\alpha$. Then
$$p_k(G)=\sum_{\alpha\in \mathcal{N}_k}N_k(G,\alpha)\cdot TW(\alpha).$$
Therefore
$$\frac{p_k(G)}{v(G)}=\sum_{\alpha\in \mathcal{N}_k}\\\mathbb{P}(G,\alpha,k)\cdot TW(\alpha).$$
Hence, if $(G_i)$ is Benjamini--Schramm convergent then for every fixed $k$, the sequence
$$\frac{p_k(G_i)}{v(G_i)}=\int z^k\, d\rho_{G_i}(z)$$
is convergent. Clearly, this implies that for every polynomial $q(z)$, the sequence
$$\int q(z)\, d\rho_{G_i}(z)$$
is convergent.

Assume that $D$ is a general upper bound for all degrees of all graphs $G_i$. Then all zeros of $\mu(G_i,x)$ lie in the interval $[-2\sqrt{D-1},2\sqrt{D-1}]$. Since every continuous function can be uniformly approximated by a polynomial on a bounded interval, we obtain that the sequence $(\rho_{G_i})$ is weakly convergent.
\end{proof}

Now we are ready to prove Theorem~\ref{entropy}.

\begin{proof}[Proof of Theorem~\ref{entropy}] First we prove part (a) and (b). For a graph $G$ let $S(G)$ denote the set of zeros of the matching polynomial $\mu(G,x)$, then
$$M(G,t)=\prod_{\lambda\in S(G) \atop \lambda>0}(1+\lambda^2t)=\prod_{\lambda\in S(G)}(1+\lambda^2t)^{1/2}.$$
Then
$$\ln M(G,t)=\sum_{\lambda\in S(G)}\frac{1}{2}\ln\left(1+\lambda^2t\right).$$
By differentiating both sides with respect to $t$ we get that
$$\frac{\frac{d}{dt}M(G,t)}{M(G,t)}=\sum_{\lambda\in S(G)}\frac{1}{2}\frac{\lambda^2}{1+\lambda^2t}.$$
Hence
$$p(G,t)=\frac{2t\cdot \frac{d}{dt}M(G,t)}{v(G)\cdot  M(G,t)}=\frac{1}{v(G)}\sum_{\lambda\in S(G)}\frac{\lambda^2 t}{1+\lambda^2t}=\int \frac{tz^2}{1+tz^2}\, d\rho_{G}(z).$$
Similarly,
$$\frac{\ln M(G,t)}{v(G)}=\frac{1}{v(G)}\sum_{\lambda\in S(G)}\frac{1}{2}\ln\left(1+\lambda^2t\right)=\int \frac{1}{2}\ln\left(1+tz^2\right)\, d\rho_{G}(z).$$
Since $(G_i)$ is Benjamini--Schramm convergent bounded degree graph sequence, the sequence $(\rho_{G_i})$ 
weakly converges to some $\rho_L$ by Theorem~\ref{wc}. Since both functions 
$$\frac{tz^2}{1+tz^2}\ \ \ \ \mbox{and}\ \ \ \  \frac{1}{2}\ln\left(1+tz^2\right)$$
are continuous, we immediately obtain that
$$\lim_{n\to \infty}p(G_i,t)=\int \frac{tz^2}{1+tz^2} \, d\rho_L(z),$$
and 
$$\lim_{n\to \infty}\frac{\ln M(G_i,t)}{v(G_i)}=\int \frac{1}{2}\ln\left(1+tz^2\right)\, d\rho_L(z).$$
Note that both functions
$$\frac{tz^2}{1+tz^2}\ \ \ \ \mbox{and}\ \ \ \  \frac{1}{2}\ln\left(1+tz^2\right)$$
are strictly monotone increasing continuous functions in the variable $t$. Thus their integrals are also strictly monotone increasing continuous functions.
\bigskip

To prove part (c), let us introduce the function
$$p(L,t)=\int \frac{tz^2}{1+tz^2} \, d\rho_L(z).$$
We have seen that $p(L,t)$ is a strictly monotone increasing continuous function, and $\lim_{i\to \infty}p(G_i,t)=p(L,t)$. 
Since for all $G_i$, $p^*(G_i)\geq p_0$, we have $\lim_{t \to \infty}p(G_i,t)\geq p_0$ for all $i$. This means that $\lim_{t \to \infty}p(L,t)\geq p_0$.
Hence we can consider the inverse function of $p(L,t)$ which maps $[0,p_0)$ into $[0,\infty)$, let us call it $t(L,p)$.
We show that 
$$\lim_{i\to \infty}t(G_i,p)=t(L,p)$$
pointwise for $p<p_0$. Assume by contradiction that this is not the case. This means that for some $p_1$, there exists an $\varepsilon$ and an infinite sequence $n_i$ for which
$$\left|t(L,p_1)-t(G_{n_i},p_1)\right|\geq \varepsilon.$$
We distinguish two cases according to \\
(i) there exists an infinite sequence $(n_i)$ for which
$$t(G_{n_i},p_1)\geq t(L,p_1)+\varepsilon,$$
or  
(ii) there exists an infinite sequence $(n_i)$ for which
$$t(G_{n_i},p_1)\leq t(L,p_1)-\varepsilon.$$
In the first case, let $t_1=t(L,p_1)$, $t_2=t_1+\varepsilon$ and $p_2=p(L,t_2)$. Clearly, $p_2>p_1$. Note that
$$t(G_{n_i},p_1)\geq t(L,p_1)+\varepsilon=t_2$$
and $p(G_{n_i},t)$ are monotone increasing functions, thus 
$$p(G_{n_i},t_2)\leq p(G_{n_i},t(G_{n_i},p_1))=p_1=p_2-(p_2-p_1)=p(L,t_2)-(p_2-p_1).$$
This contradicts the fact that 
$$\lim_{n_i\to \infty}p(G_{n_i},t_2)=p(L,t_2).$$
In the second case, let $t_1=t(L,p_1)$, $t_2=t_1-\varepsilon$ and $p_2=p(L,t_2)$. Clearly, $p_2<p_1$. Note that
$$t(G_{n_i},p_1)\leq t(L,p_1)-\varepsilon=t_2$$
and $p(G_{n_i},t)$ are monotone increasing functions, thus 
$$p(G_{n_i},t_2)\geq p(G_{n_i},t(G_{n_i},p_1))=p_1=p_2+(p_1-p_2)=p(L,t_2)+(p_1-p_2).$$
This again contradicts the fact that $$\lim_{n\to \infty}p(G_{n_i},t_2)=p(L,t_2).$$
Hence $\lim_{i\to \infty}t(G_i,p)=t(L,p)$. 
\medskip

Finally, we show that $\lambda_{G_i}(p)$ converges for all $p$. Let $t=t(L,p)$, and 
$$\lambda_L(p)=\lim_{i\to \infty}\frac{\ln M(G_i,t)}{v(G_i)}-\frac{1}{2}p\ln(t).$$
Note that
$$\lambda_{G_i}(p)=\frac{\ln M(G_i,t_i)}{v(G_i)}-\frac{1}{2}p\ln(t_i),$$
where $t_i=t(G_i,p)$. We have seen that $\lim_{i\to \infty}t_i=t$. Hence it is enough to prove that the functions
$$\frac{\ln M(G_i,u)}{v(G_i)}$$
are equicontinuous. Let us fix some positive $u_0$ and let
$$R(u_0,u)=\max_{z\in [-2\sqrt{D-1},2\sqrt{D-1}]}\left|\frac{1}{2}\ln\left(1+u_0z^2\right)-\frac{1}{2}\ln\left(1+uz^2\right)\right|.$$
Clearly, if $|u-u_0|\leq \delta$ for some sufficiently small $\delta$, then $R(u_0,u)\leq \varepsilon$, and
$$\left|\frac{\ln M(G_i,u)}{v(G_i)}-\frac{\ln M(G_i,u_0)}{v(G_i)}\right|=\left|\int \frac{1}{2}\ln\left(1+u_0z^2\right)\, d\rho_{G_i}(z)-\int \frac{1}{2}\ln\left(1+uz^2\right)\, d\rho_{G_i}(z)\right|\leq$$
$$\leq \int \left|\frac{1}{2}\ln\left(1+u_0z^2\right)-\frac{1}{2}\ln\left(1+uz^2\right)\right|\, d\rho_{G_i}(z)\leq \int R(u,u_0)\, d\rho_{G_i}(z)\leq \varepsilon.$$
This completes the proof of the convergence of $\lambda_{G_i}(p)$.
\end{proof}

Now it is easy to define these functions for those random rooted graphs which can be obtained as a Benjamini--Schramm limit of finite graphs.

\begin{Def} Let $L$ be a random rooted graph which can be obtained as a limit of Benjamini--Schramm limit of finite graphs $(G_i)$ of bounded degree. Assume that $p^*(G_i)\geq p_0$ for all $n$. Let
$$p(L,t)=\lim_{n\to \infty}p(G_i,t),$$
$$F(L,t)=\lim_{n\to \infty} \frac{\ln M(G_i,t)}{v(G_i)}$$
for all $t\geq 0$, and 
$$t(L,p)=\lim_{n\to \infty}t(G_i,p),$$
$$\lambda_L(p)=\lim_{n\to \infty}\lambda_{G_i}(p)$$
for all $p<p_0$. Finally, let
$$\lambda_L(p_0)=\lim_{p\nearrow p_0}\lambda_L(p).$$
\end{Def}

Note that the functions $p(L,t),F(L,t),t(L,p)$ and $\lambda_L(p)$ are well-defined in the sense that if the sequences $(G_i)$ and $(H_i)$ both Benjamini--Schramm converge to $L$ such that $p^*(G_i),p^*(H_i)\geq p_0$ for all $i$, then they define the same functions on $[0,\infty)$ or $[0,p_0]$. Indeed, we can consider the two sequences together and apply Theorem~\ref{entropy} to obtain that the limits do not depend on the choice of the sequence. From the proof of Theorem~\ref{entropy}, we also see that $p(L,t)$ and $F(L,t)$ can be expressed as integrals along a certain measure $\rho_L$.

\subsection{Entropy and density function of the infinite $d$--regular tree $\mathbb{T}_d$} 

In this section we give the entropy and density functions of the $d$--regular and $(a,b)$-biregular trees. 

\begin{Th} \label{tree} Let $\mathbb{T}_d$ be the infinite $d$--regular tree. Then \\
(a) $$p(\mathbb{T}_d,t)=\frac{2d^2t+d-d\cdot\sqrt{1+4(d-1)t}}{2d^2t+2}.$$
(b) $$\int \frac{1}{2}\ln (1+tz^2)\, d \rho_{\mathbb{T}_d}(z)=\frac{1}{2}\ln S_d(t),$$
where 
$$S_d(t)=\frac{1}{\eta_t^2}\left(\frac{d-1}{d-\eta_t}\right)^{d-2},$$
where 
$$\eta_t=\frac{\sqrt{1+4(d-1)t}-1}{2(d-1)t}.$$
(c) $$t(\mathbb{T}_d,p)=\frac{p(d-p)}{d^2(1-p)^2}.$$
(d) $$\lambda_{\mathbb{T}_d}(p)=\G_d(p)=\frac{1}{2}\left(p\ln \left(\frac{d}{p}\right)+(d-p)\ln \left(1-\frac{p}{d}\right)-2(1-p)\ln (1-p)\right).$$
\end{Th}

\begin{Th} \label{tree2} Let $\mathbb{T}_{a,b}$ be the infinite $(a,b)$-biregular tree. Then for $0\leq p\leq \min(\frac{2a}{a+b},\frac{2b}{a+b})$ we have \\
(a) $$p(\mathbb{T}_{a,b},t)=\frac{2abt+\frac{2ab}{a+b}-\frac{2ab}{a+b}\sqrt{1+(2a+2b-4)t+(a-b)^2t^2}}{2abt+2}.$$
(b) $$t(\mathbb{T}_{a,b},p)=\frac{a+b}{2ab}\frac{p\left(1-\frac{a+b}{2ab}p\right)}{\left(1-\frac{a+b}{2a}p\right)\left(1-\frac{a+b}{2b}p\right)}.$$
(c) $$\lambda_{\mathbb{T}_{a,b}}(p)=\G_{a,b}(p)=\frac{a}{a+b}H\left(\frac{a+b}{2a}p\right)+\frac{b}{a+b}H\left(\frac{a+b}{2b}p\right)+\frac{1}{2}p\ln (ab)-\frac{ab}{a+b}H\left(\frac{a+b}{2ab}p\right),$$
where $H(q)=-(q\ln (q)+(1-q)\ln (1-q))$. 
\end{Th}

There are two essentially different proofs for Theorem~\ref{tree} and \ref{tree2}. We detail the first proof, and in the next subsection we sketch a second incomplete one.

The first proof of Theorem~\ref{tree} and \ref{tree2}  roughly follows the arguments of Section 4 of \cite{ACFK}. For an (infinite)   tree, the spectral measure and the matching measure coincide. This can be proved via Lemma 4.2 of \cite{ACFK}, or an even simpler proof is that for trees, the number of closed walks and the number of closed tree-like walks are the same, so the moment sequences of the spectral measure and the matching measure coincide and since they are supported on a bounded interval, they must be the same measure. For the $d$--regular tree $\mathbb{T}_d$, this is the Kesten-McKay measure given by the density function
$$f_d(x)=\frac{d \sqrt{4(d-1)-x^2}}{2\pi (d^2-x^2)}\chi_{[-2\sqrt{d-1},2\sqrt{d-1}]}.$$
For the $(a,b)$--biregular infinite tree, the matching or spectral measure $\rho_{\mathbb{T}_{a,b}}$ is given by
$$d\rho_{\mathbb{T}_{a,b}}=\frac{|a-b|}{a+b}\delta_0+\frac{ab\sqrt{-(x^2-ab+(s-1)^2)(x^2-ab+(s+1)^2)}}{\pi (a+b)(ab-x^2)|x|}\chi_{\{|\sqrt{a-1}-\sqrt{b-1}|\leq |x|\leq \sqrt{a-1}+\sqrt{b-1}\}}d\, x,$$
where $s=\sqrt{(a-1)(b-1)}$. As a next step one might try to compute the integral of the functions
$$\frac{tz^2}{1+tz^2}\ \ \ \ \mbox{and}\ \ \ \  \frac{1}{2}\ln\left(1+tz^2\right)$$
to obtain $p(\mathbb{T}_{a,b},t)$ and $F(\mathbb{T}_{a,b},t)$. We will slightly modify this argument to simplify it. Our modification yields that we do not need to compute these integrals, we can work directly with the moment sequences which are simply the number of closed walks in the corresponding trees. More precisely, in $\mathbb{T}_{a,b}$ we have to weight the number of closed walks starting and ending at a root vertex of degree $a$ with weight $\frac{b}{a+b}$, and the  number of closed walks starting and ending at a root vertex of degree $b$ with weight $\frac{a}{a+b}$.

First of all, we need the following lemma on the number of closed walks of $\mathbb{T}_{a,b}$. The current author is completely sure that it is well-known, but since we were not able to find any reference, we give its proof.

\begin{Lemma} Let $W_j^a$ and $W_j^b$ be the number of closed walks of length $j$ starting and returning to a root vertex of $\mathbb{T}_{a,b}$ of degree $a$, and of degree $b$, respectively. Then the generating function
$$G_a(z):=\sum_{j=0}^{\infty}W_j^az^j=\frac{1}{1-az^2F_b(z)},$$
where 
$$F_b(z)=\frac{1+(b-a)z^2-\sqrt{1-(2a+2b-4)z^2+(b-a)^2z^4}}{2(a-1)z^2}.$$
Similarly,
$$G_b(z):=\sum_{j=0}^{\infty}W_j^bz^j=\frac{1}{1-bz^2F_a(z)},$$
where 
$$F_a(z)=\frac{1+(a-b)z^2-\sqrt{1-(2a+2b-4)z^2+(b-a)^2z^4}}{2(b-1)z^2}.$$
\end{Lemma}

\begin{proof} Let us consider the rooted tree $\mathbb{T}^a_{a,b}$, where the only difference compared to $\mathbb{T}_{a,b}$ that the root vertex has degree $a-1$ and not $a$. Similarly, let the rooted tree $\mathbb{T}^b_{a,b}$ be the tree where the only difference compared to $\mathbb{T}_{a,b}$ that the root vertex has degree $b-1$ and not $b$. Let $\overline{W}_j^a$ be the number of closed walks of length $j$ starting and returning to the root vertex $\mathbb{T}^a_{a,b}$. Furthermore, let  $\overline{U}_j^a$ be the number of closed walks of length $j$ starting and returning to the root vertex $\mathbb{T}^a_{a,b}$ such that the walk only visits the root at the beginning and at the end, and the walk has length at least $2$, so it is different form the empty walk. We can similarly define
$\overline{W}_j^b$ and $\overline{U}_j^b$.
Let
$$F_a(z)=\sum_{j=0}^{\infty}\overline{W}_j^az^j\ \ \mbox{and},\  \ F_b(z)=\sum_{j=0}^{\infty}\overline{W}_j^bz^j,$$
and
$$R_a(z)=\sum_{j=1}^{\infty}\overline{U}_j^az^j\ \  \mbox{and},\  \ R_b(z)=\sum_{j=1}^{\infty}\overline{U}_j^bz^j.$$
First of all,
$$F_a(z)=1+R_a(z)+R_a(z)^2+R_a(z)^3+\dots=\frac{1}{1-R_a(z)},$$
since every closed walk can be decomposed uniquely to walks which visit the root only at the beginning and at the end.
Similarly,
$$F_b(z)=\frac{1}{1-R_b(z)}.$$
Finally,
$$R_a(z)=(a-1)z^2F_b(z)\ \ \mbox{and similarly},\ \ R_b(z)=(b-1)z^2F_a(z),$$
since every closed walk which visits the root only at the beginning and at the end can be decomposed in the following way: we erase the first and last step (we can choose this $a-1$ different ways in $\mathbb{T}^a_{a,b}$), and we get a closed walk in $\mathbb{T}^b_{a,b}$.
Solving these equation we get that
$$F_a(z)=\frac{1+(b-a)z^2-\sqrt{1-(2a+2b-4)z^2+(b-a)^2z^4}}{2(b-1)z^2},$$
$$F_b(z)=\frac{1+(a-b)z^2-\sqrt{1-(2a+2b-4)z^2+(b-a)^2z^4}}{2(a-1)z^2},$$
$$R_a(z)=\frac{1}{2}\left(1+(a-b)z^2-\sqrt{1-(2a+2b-4)z^2+(b-a)^2z^4}\right),$$
$$R_b(z)=\frac{1}{2}\left(1+(b-a)z^2-\sqrt{1-(2a+2b-4)z^2+(b-a)^2z^4}\right).$$
(Note that at some point, we have to solve a quadratic equation, and we can choose only the minus sign, because of the evaluation of the generating function at the value $z=0$.)
\medskip

Now let us turn back to the original problem. Let $U_j^a$ be the number of closed walks of length $j$ starting and returning to the root vertex $\mathbb{T}_{a,b}$ of degree $a$ such that the walk only visits the root at the beginning and at the end, and the walk has length at least $2$, so it is different form the empty walk. We define $U_j^b$ similarly. Let
$$G_a(z)=\sum_{j=0}^{\infty}W_j^az^j\ \ \mbox{and},\ \ G_b(z)=\sum_{j=0}^{\infty}W_j^bz^j,$$
and
$$H_a(z)=\sum_{j=1}^{\infty}U_j^az^j\ \ \mbox{and},\ \ H_b(z)=\sum_{j=1}^{\infty}U_j^bz^j.$$
As before,
$$G_a(z)=\frac{1}{1-H_a(z)}\ \ \mbox{and},\ \ G_b(z)=\frac{1}{1-H_b(z)}.$$
Finally,
$$H_a(z)=az^2F_b(z)\ \ \mbox{and similarly},\ \ R_b(z)=bz^2F_a(z)$$
since every closed walk which visits the root only at the beginning and at the end can be decomposed in the following way: we erase the first and last step (we can choose this $a$ different ways in $\mathbb{T}_{a,b}$), and we get a closed walk in $\mathbb{T}^b_{a,b}$.
Hence
$$G_a(z)=\frac{1}{1-az^2F_b(z)}\ \ \mbox{and},\ \ G_b(z)=\frac{1}{1-bz^2F_a(z)}.$$
\end{proof}

\begin{proof}[Proof of Theorem~\ref{tree} and \ref{tree2}.] Since Theorem~\ref{tree} is a special case of Theorem~\ref{tree2}, we will concentrate on the proof of the latter theorem. We only need to work with part (a), since then part (b) follows immediately and part (c) follows from part (b) using
$$\frac{d}{dp}\lambda_{\mathbb{T}_{a,b}}(p)=-\frac{1}{2}\ln t(\mathbb{T}_{a,b},p).$$
These are routine computations which we left to the Reader.

To prove part (a), first let us assume that $|t|<\frac{1}{4(\max(a,b)-1)}$. Note that for such a $t$, all subsequent series are converging. Then
$$p(\mathbb{T}_{a,b},t)=\int \frac{tu^2}{1+tu^2}\, d\rho_{\mathbb{T}_{a,b}}(u)=\int \sum_{j=1}^{\infty}(-1)^{j+1}t^ju^{2j}\, d\rho_{\mathbb{T}_{a,b}}(u)=$$
$$=\sum_{j=1}^{\infty}(-1)^{j+1}t^j\int u^{2j}\, d\rho_{\mathbb{T}_{a,b}}(u).$$
Note that 
$$\int u^{2j}\, d\rho_{\mathbb{T}_{a,b}}(u)=\frac{b}{a+b}W_{2j}^a+\frac{a}{a+b}W_{2j}^b.$$
Hence
$$p(\mathbb{T}_{a,b},-z^2)=1-\left(\frac{b}{a+b}G_a(z)+\frac{a}{a+b}G_b(z)\right).$$
After some calculation we get that
$$p(\mathbb{T}_{a,b},t)=\frac{2abt+\frac{2ab}{a+b}-\frac{2ab}{a+b}\sqrt{1+(2a+2b-4)t+(a-b)^2t^2}}{2abt+2}$$
for $|t|<\frac{1}{4(\max(a,b)-1)}$. On the other hand, both functions appearing in the previous equation are holomorhic in the region $\{t\ |\ |\Im (t)|\leq \Re (t)\}$, so they must be the same everywhere in this region.
\end{proof}

\subsection{Random graphs.} The goal of this subsection is twofold. On the one hand, we show that Theorem~\ref{main} and \ref{biregular} are quite precise, for instance if $p$ is separated away from $1$ then Theorem~\ref{main} is the best possible up to a constant factor. On the other hand, we also would like to show a connection between random 
(bi)regular random graphs and the entropy function of an infinite (bi)regular tree.
\medskip

An alternative way to obtain Theorem~\ref{tree} and \ref{tree2} is the following. We can use Theorem~\ref{entropy} to obtain the required functions by choosing an appropriate Benjamini--Schramm convergent graph sequence. It turns out that it is sufficient to consider random $d$--regular or $(a,b)$-biregular bipartite graphs. Indeed, one can compute the expected number of $k$-matchings of a random $d$--regular or $(a,b)$-biregular bipartite graphs quite easily. Such a computation was carried out in \cite{BM,FKM,sch2,GJR} for $d$--regular bipartite graphs and it easily generalizes to  $(a,b)$-biregular bipartite graphs. We also note that a random $(a,b)$-biregular bipartite graph contains very small number of short cycles. This is a classical result for random regular graphs, but it is also known for biregular bipartite graphs \cite{DJ}.

First of all, let us specify which biregular random graph model we use. Let the vertex set of the random graph be $V\cup W$, where $V=\{v_1,\dots ,v_{an}\}$ and $W=\{w_1,\dots ,w_{bn}\}$. Let us consider two random partition of the set $\{1,\dots ,abn\}$, the first one $P_1=\{A_1,\dots ,A_{an}\}$ where each set has size $b$, and a second one $P_2=\{B_1,\dots ,B_{bn}\}$ where each set has size $a$. Then for every $k\in \{1,\dots ,abn\}$ connect 
$v_i$ and $w_j$ if $k\in A_i\cap B_j$. This is the configuration model. Note that this model allows multiple edges, but it is not a problem for us. In the special case when $a=b=d$ we can choose $V$ and $W$ to be of size $n$. The following theorem was proved in \cite{FKM}.

\begin{Th} \cite{FKM} \label{random2} Let $G$ be chosen from the set of labelled $d$--regular bipartite graphs on $v(G)=2n$ vertices according to the configuration model. Then
$$\mathbb{E}m_k(G)={n \choose k}^2d^{2k}\frac{1}{{dn \choose k}}.$$
\end{Th} 

The corollary of this theorem is the second part of Theorem~\ref{main}.

\begin{Cor} Let $p=\frac{k}{n}$. There exists a $d$--regular bipartite graph $G$ on $2n$ vertices such that
$$m_k(G)\leq \sqrt{\frac{1-p/d}{1-p}}\cdot p_{\mu}\cdot \exp(2n \G_d(p)).$$
\end{Cor}

\begin{proof} We will show that
$$E={n \choose k}^2d^{2k}\frac{1}{{dn \choose k}}\leq \sqrt{\frac{1-p/d}{1-p}}\cdot p_{\mu}\cdot \exp(2n \G_d(p)).$$
Note that
$$E={n \choose k}^2d^{2k}\frac{1}{{dn \choose k}}=\frac{{n \choose k}}{{dn \choose k}}d^{2k}{n \choose k}=\frac{{n \choose k}}{{dn \choose k}}d^{2k} \cdot \frac{p_{\mu}}{p^k(1-p)^{n-k}}.$$
For the first term we use Stirling's formula.
Let $\Theta_m$ be defined by the following form of  Stirling's formula: 
$$m!=\sqrt{2\pi m} \left(\frac{m}{e}\right)^m e^{\Theta_m}.$$
It is known (see \cite{rob}) that
$$\frac{1}{12m+1}\leq \Theta_m\leq \frac{1}{12m}.$$
Then
$$E=\frac{1}{\sqrt{2\pi \frac{k(n-k)}{n}}}e^{\Theta_n-\Theta_k-\Theta_{n-k}}\sqrt{2\pi \frac{k(dn-k)}{dn}}e^{-\Theta_{dn}+\Theta_k+\Theta_{dn-k}}\cdot p_{\mu}\cdot \exp(2n \G_d(p))=$$
$$=\sqrt{\frac{1-p/d}{1-p}}e^{\Theta_n-\Theta_{n-k}-\Theta_{dn}+\Theta_{dn-k}}\cdot p_{\mu}\cdot \exp(2n \G_d(p)).$$
Thus we only need to show that
$$\Theta_n-\Theta_{n-k}-\Theta_{dn}+\Theta_{dn-k}\leq 0.$$
This is indeed true:
$$\Theta_n-\Theta_{n-k}-\Theta_{dn}+\Theta_{dn-k}\leq \frac{1}{12n}-\frac{1}{12(n-k)+1}-\frac{1}{12dn+1}+\frac{1}{12(dn-k)}=$$
$$=\frac{-(12k-1)}{12n(12(n-k)+1)}+\frac{12k+1}{(12dn+1)(12(dn-k)+1)}\leq 0$$
if $d\geq 2$.
\end{proof}

The following lemma is a straightforward extension of the previous results to $(a,b)$--biregular bipartite graphs.

\begin{Lemma} \label{random} Let $G$ be chosen from the set of labelled $(a,b)$--biregular bipartite graphs on $v(G)=(a+b)n$ vertices according to the configuration model.
\medskip

\noindent (a) Then
$$\mathbb{E}m_k(G)=\exp(v(G)(\G_{a,b}(p)+o_{v(G)}(1))),$$
where $p=2k/v(G)$. 
\medskip

\noindent (b) \cite{DJ} Let $c_{2j}(G)$ be the number of $2j$-cycles in the graph $G$. Then
$$\mathbb{E}c_{2j}(G)=\frac{((a-1)(b-1))^{j}}{2j}(1+o_{v(G)}(1)).$$
\end{Lemma}

\begin{proof} (a) Note that the number of all partitions pairs $(P_1,P_2)$ is
$$N=\frac{(abn)!}{a!^{bn}}\cdot\frac{(abn)!}{b!^{an}}.$$
The number of possible $k$-matchings is
$$U_k={abn \choose k}{an \choose k}{bn \choose k}k!^2.$$
If we fix one $k$-matching then we need to repartition the remaining $(abn-k)$ elements into sets of sizes $a$ and $a-1$, and $b$ and $b-1$. This can be done in
$$V_k=\frac{(abn-k)!}{(a-1)!^k a!^{bn-k}}\cdot \frac{(abn-k)!}{(b-1)!^k b!^{an-k}}$$
ways.
Hence
$$\mathbb{E}m_k(G)=\frac{1}{N}U_kV_k={an \choose k}{bn \choose k}(ab)^k\frac{1}{{abn \choose k}}.$$
Then by the usual approximation of binomial coefficients we get that
$$\mathbb{E}m_k(G)=\exp(v(G)(\G_{a,b}(p)+o_{v(G)}(1))),$$
where $p=2k/v(G)$.
\medskip

\noindent (b) We can choose the possible cycles in
$$T_j={abn \choose 2j}{an \choose j}{bn \choose j}(2j-1)! j!^2$$
different ways. (We can choose the 'edges', and vertices in ${abn \choose 2j}{an \choose j}{bn \choose j}$ ways, then we choose an ordering on the edges, and on each vertex sets, and we connect the vertices and 'edges' along the orderings. Finally, we divide by $(2j)$ since we counted each cycles in $2j$ ways.) Next we need to repartition the remaining $(abn-2j)$ elements into sets of sizes $a$ and $a-2$, and $b$ and $b-2$. This can be done in
$$S_j=\frac{(abn-2j)!}{(a-2)!^j a!^{bn-j}}\cdot \frac{(abn-2j)!}{(b-2)!^j b!^{an-j}}$$
ways. Hence
$$\mathbb{E}c_{2j}(G)=\frac{1}{N}T_jS_j=\frac{((a-1)(b-1))^{j}}{2j}(1+o_{v(G)}(1)).$$
\end{proof}

Part (b) of Lemma~\ref{random} shows that the expected number of cycles of length $2j$ is bounded independently of the size of the graph. Note that the $(a,b)$-biregular graph sequence $(G_i)$ Benjamini--Schramm converges to $\mathbb{T}_{a,b}$ if for all fixed $j$ we have $c_{2j}(G_i)=o(v(G_i))$.  Note that by Markov's inequality:
$$\mathbb{P}(m_k(G)>3\mathbb{E}m_k(G))\leq \frac{1}{3}\ \  \mbox{and,}\ \ \ \mathbb{P}(c_{2j}(G)>3g \mathbb{E}c_{2j}(G))\leq \frac{1}{3g}$$
for $j=1,\dots ,g$. Hence for any large enough $n$ and fixed $g$, with probability at least $1/3$ we can choose a graph $G_i$ on $(a+b)n$ vertices such that
$G_i$ has a bounded number of cycles of length at most $2g$ and $m_k(G_i)\leq 3\exp(v(G)(\G_{a,b}(p)+o_{v(G)}(1)))$. This shows that we can choose a sequence of graphs $(G_i)$ converging to $\mathbb{T}_{a,b}$ such that
$$\frac{\ln m_k(G_i)}{v(G_i)}+o_{v(G_i)}(1)=\lambda_{G_i}(p)\leq \G_{a,b}(p)+o_{v(G_i)}(1).$$
This implies that
$$\lambda_{\mathbb{T}_{a,b}}(p)\leq \G_{a,b}(p).$$
Note that we only proved this inequality for rational $p$, but then it follows for all $p$ by continuity.
\bigskip

Unfortunately, with this idea we were not able to establish the inequality $\lambda_{\mathbb{T}_{a,b}}(p)\geq \G_{a,b}(p)$.
The problem is the following. In principle, it can occur that a typical random graph has much smaller (exponentially smaller) number of $k$-matchings than the expected value, and a large contribution to the expected value comes from graphs having large number of short cycles and matchings. Note that Theorem~\ref{biregular} implies that this cannot occur, but we cannot use this result as it would result a cycle in the proof of this theorem. Instead we propose a conjecture which would imply the inequality $\lambda_{\mathbb{T}_{a,b}}(p)\geq \G_{a,b}(p)$.

\begin{Conj} There exists a constant $C$ independently of $n$ and $k$ such that
$$\mathbb{E}m_k(G)^2\leq C (\mathbb{E}m_k(G))^2.$$
\end{Conj}

Note that this conjecture is known to be true for perfect matchings in regular random graphs \cite{BM}.
To show that this conjecture implies $\lambda_{\mathbb{T}_{a,b}}(p)\geq \G_{a,b}(p)$, we need the following proposition.

\begin{Prop} Let $X$ be a non-negative random variable such that for some positive constant $C$ we have
$$\mathbb{P}(X> C\cdot \mathbb{E}X)\leq \frac{1}{16C}\ \ \mbox{and}\ \ \ \mathbb{E}X^2\leq C (\mathbb{E}X)^2.$$
Then 
$$\mathbb{P}\left(\frac{1}{4}\mathbb{E}X\leq X\leq C\mathbb{E}X\right)\geq \frac{1}{2C}.$$
\end{Prop}

\begin{proof} Let $A=\{\omega\ |\ X(\omega)<  \frac{1}{4}\mathbb{E}X\}$, $B=\{\omega\ |\ \frac{1}{4}\mathbb{E}X\leq X(\omega)\leq  \ C\mathbb{E}X\}$, and $D=\{\omega\ |\ X(\omega)> C\mathbb{E}X\}$. Then
$$\int_A X dP\leq \frac{1}{4}\mathbb{E}X.$$
Furthermore,
$$\mathbb{P}(D)\cdot \mathbb{E}X^2\geq \mathbb{P}(D)\cdot \int_D X^2 dP=\int_D 1 dP \cdot \int_D X^2 dP\geq \left( \int_D X dP\right)^2.$$
Hence
$$\frac{1}{16C} C (\mathbb{E}X)^2\geq \mathbb{P}(D)\cdot \mathbb{E}X^2\geq \left( \int_D X dP\right)^2.$$
In other words,
$$\int_D X dP\leq \frac{1}{4}\mathbb{E}X.$$
This implies that
$$\int_B X dP\geq \frac{1}{2}\mathbb{E}X.$$
Since 
$$\int_B X dP \leq \mathbb{P}(B) C\mathbb{E}X,$$
the claim of the proposition follows immediately.
\end{proof}

Let us fix a positive number $g$, and let us call a graph \emph{typical} if 
$$c_{2j}(G)<16Cg \mathbb{E}c_{2j}(G),$$
for $j=1,\dots ,g$. Note that a typical graph has bounded number of short cycles and by Markov's inequality, the probability that a graph is typical is at least $1-\frac{1}{16C}$. First case: we find a typical graph $G$ such that $m_k(G)>C \mathbb{E}m_k(G)$, then we are done, because then $\lambda_G(p)\geq \G_{a,b}(p)+o(1)$. Second case: there is no typical graph with $m_k(G)>C \mathbb{E}m_k(G)$, then the proposition implies that the probability 
$$\mathbb{P}\left(\frac{1}{4}\mathbb{E}m_k(G)\leq m_k(G)\leq C\mathbb{E}m_k(G)\right)\geq \frac{1}{2C}.$$
Since the  probability that a graph is typical is at least $1-\frac{1}{16C}$, we see that there are typical graphs for which 
$$m_k(G)\geq \frac{1}{4}\mathbb{E}m_k(G)$$
implying again that $\lambda_G(p)\geq \G_{a,b}(p)+o(1)$. Hence we can choose a sequence of typical graphs to show that $\lambda_{\mathbb{T}_{a,b}}(p)\geq \G_{a,b}(p)$.

In spite of the fact that this proof did not lead to another proof of Theorem~\ref{tree2}, we feel that it was instructive to carry out these computations as they showed that Theorem~\ref{main} and Theorem~\ref{biregular} are tight. This was known for perfect matchings of $d$--regular random graphs \cite{BM,sch2}, and for matchings of arbitrary size \cite{FKM}. Our computation for biregular bipartite graphs is the natural counterpart of these results.

\section{New version of Gurvits's theorem} \label{effective form}

In this section we prove the following theorem.

\begin{Th} \label{equivalent} The following two statements are equivalent. 
\medskip

\noindent (i) For any  $d$--regular bipartite graph $G$ on $2n$ vertices, we have
$$\frac{\ln m_k(G)}{v(G)}\geq \G_d(p)+o_{v(G)}(1),$$
where $p=\frac{k}{n}$ and $m_k(G)$ denotes the number of matchings of size $k$.
\medskip

\noindent (ii) For any $d$--regular bipartite graph $G$, we have
$$\lambda_G(p)\geq \G_d(p).$$
\end{Th}

\begin{proof} First we show that (i) implies (ii). Since both functions $\lambda_G(p)$ and $\G_d(p)$
are continuous, it is enough to prove the claim for rational numbers $p$. Let $p=\frac{a}{b}$. Let us consider $br$ copies of $G$, and let us consider the matchings of size $k=ar$. Then
$$\lambda_G(p)=\lambda_{brG}(p)\geq \frac{\ln m_k(brG)}{v(brG)}-\frac{\ln v(brG)}{v(brG)}\geq \G_d(p)+o_{v(brG)}(1)-\frac{\ln v(brG)}{v(brG)}.$$
The (first) equality follows from part (a) of Proposition~\ref{asymp}, the first inequality follows from part (d) of Proposition~\ref{asymp}, the second inequality is the assumption of (i). As $r$ tends to infinity, the last two terms disappear, and  we get that
$$\lambda_G(p)\geq \G_d(p).$$
\medskip

\noindent Next we show that (ii) implies (i).
$$\frac{\ln m_k(G)}{v(G)}\geq \lambda_G(p)-\frac{\ln v(G)}{v(G)}\geq \G_d(p)-\frac{\ln v(G)}{v(G)}.$$
The first inequality follows from part (d) of Proposition~\ref{asymp}, the second inequality is the assumption of (ii).
Note that $-\frac{\ln v(G)}{v(G)}=o_{v(G)}(1)$. So we are done.
\end{proof}

\begin{Cor} Theorem~\ref{Gurvits} implies that
$$\frac{\ln m_k(G)}{v(G)}\geq \G_d(p)-\frac{\ln v(G)}{v(G)}.$$
\end{Cor}

\begin{proof} See the second part of the proof of Theorem~\ref{equivalent}.
\end{proof}

\begin{Rem} \label{rem: Gurvits} L. Gurvits actually proved much stronger results than Theorem~\ref{Gurvits}. 

He proved that for all pairs of $n \times n$ matrices $(P,Q)$, where $P$ is nonnegative and $Q$ is doubly stochastic we have
$$\ln(\mathrm{Per}(P)) \geq \sum_{1 \leq i,j \leq n} (1- Q(i,j))\ln(1- Q(i,j))-\sum_{1 \leq i,j \leq n} Q(i,j) \ln \left(\frac{Q(i,j)}{P(i,j)} \right).$$
From this L. Gurvits deduced the following inequality: for any doubly stochastic matrix $A$ we have
$$\mathrm{Per}(A) \geq \prod_{1 \leq i,j \leq n} \left(1- A(i,j)\right)^{1- A(i,j)}.$$
Next he showed that this inequality implies that for a $d$--regular bipartite graph $G$ we have
$$m_k(G)\geq \frac{\left(1-\frac{p}{d}\right)^{\left(1-\frac{p}{d}\right)nd}\left(1-\frac{1}{n}\right)^{\left(1-\frac{1}{n}\right)2n^2(1-p)}}{\left(\frac{p}{d}\right)^{np}n^{-2n(1-p)}((n(1-p))!)^2},$$
where $p=\frac{k}{n}$ as before. For fixed $p\in (0,1)$ this gives the inequality
$$m_k(G)\geq \left(1+O\left(\frac{1}{n}\right)\right)\frac{e^{1-p}}{2\pi n(1-p)}\exp(2n\G_d(p)).$$

Let us mention that M. Lelarge \cite{Lelarge} was able to give new proofs to Gurvits's results and extending both Gurvits's results and the results in this paper by combining the methods of this paper together with new ideas.

\end{Rem}

\section{New proof of Gurvits's and Schrijver's theorems} \label{new proof}

In this section we give a new proof  of Gurvits's and Schrijver's theorems. We will prove that for any $d$--regular bipartite graph $G$, we have
$$\lambda_G(p)\geq \G_d(p).$$
According to Theorem~\ref{equivalent}, this is equivalent with Gurvits's theorem. For $p=1$ we get back Schrijver's theorem via
part (e) of Proposition~\ref{asymp}. Note that the function on the right hand side is nothing else than $\lambda_{\mathbb{T}_d}(p)$ according to Theorem~\ref{tree}.

\begin{Def} \label{def: 2-lift} Let $G$ be a graph. Then $H$ is a $2$-lift of $G$, if $V(H)=V(G)\times \{0,1\}$, and for every $(u,v)\in E(G)$, exactly one of the following two pairs are edges of $H$: $((u,0),(v,0))$ and $((u,1),(v,1))\in E(H)$ or $((u,0),(v,1))$ and $((u,1),(v,0))\in E(H)$. If $(u,v)\notin E(G)$, then none of $((u,0),(v,0))$, $((u,1),(v,1))$, $((u,0),(v,1))$ and $((u,1),(v,0))$ are edges in $H$.
\end{Def}

Note that if $G$ is bipartite then any $2$-lift of $G$ is bipartite too.

\begin{Lemma} \label{mainlemma} Let $G$ be a bipartite graph, and $H$ be a $2$-lift of $G$. Then for any $k$, we have
$$m_k(G\cup G)\geq m_k(H).$$
In particular, for any $t\geq 0$ we have
$$M(G,t)^2\geq M(H,t).$$
\end{Lemma}

\begin{proof} Since $M(G\cup G,t)=M(G,t)^2$, the inequality $m_k(H)\leq m_k(G\cup G)$ would indeed imply the second part of the lemma.
Note that $G\cup G$ can be considered as a trivial $2$-lift of $G$. Let $M$ be a matching of a $2$-lift of $G$. Let us consider the projection of $M$ to $G$, then it will consist of disjoint unions of cycles of even lengths (here we use that $G$ is bipartite!), paths and "double-edges" (i.e, when two edges project to the same edge). Let $\mathcal{R}$ be the set of these configurations. Then 
$$m_k(H)=\sum_{R \in \mathcal{R}}|\phi_H^{-1}(R)|$$
and
$$m_k(G\cup G)=\sum_{R \in \mathcal{R}}|\phi_{G\cup G}^{-1}(R)|,$$
where $\phi_H$ and $\phi_{G\cup G}$ are the projections from $H$ and $G\cup G$ to $G$. Note that
$$|\phi_{G\cup G}^{-1}(R)|=2^{k(R)},$$
where $k(R)$ is the number of connected components of $R$ different from a double-edge. On the other hand,
$$|\phi_{H}^{-1}(R)|\leq 2^{k(R)},$$
since in each component if we know the inverse image of one edge then we immediately know the inverse images of all other edges.
The only reason why there is no equality in general is that not necessarily every cycle can be obtained as a projection of a matching of a 2-lift: for instance, if one consider an 8-cycle as a 2-lift of a 4-cycle, then no matching will project to the whole 4-cycle.
Hence 
$$|\phi_{H}^{-1}(R)|\leq |\phi_{G\cup G}^{-1}(R)|$$
and consequently,
$$m_k(H)\leq m_k(G\cup G).$$
\end{proof}

By part (g) of Proposition~\ref{asymp} we get the following corollary.

\begin{Cor} \label{mainlemmacor} If $G$ is a bipartite graph, and $H$ is a $2$-lift of $G$, then $\lambda_{G}(p)\geq \lambda_{H}(p)$
for every $0\leq p\leq 1$.
\end{Cor}

\begin{Lemma}(Nathan Linial \cite{nati}) For any graph $G$, there exists a graph sequence $(G_i)_{i=0}^{\infty}$ such that $G_0=G$, $G_i$ is a $2$-lift of $G_{i-1}$ for $i\geq 1$, and $g(G_i)\to \infty$, where $g(H)$ is the girth of the graph $H$, i. e., the length of the shortest cycle.
\end{Lemma}

\begin{proof} We will show that there exists a sequence $(G_i)$ of $2$-lifts such that for any $k$, there exists an $N(k)$ such that for $j>N(k)$, the graph $G_j$ has no cycle of length at most $k$. Clearly, if $H$ has no cycle of length at most $k-1$, then any $2$-lift of it has the same property. So it is enough to prove that if $H$ has no cycle of length at most $k-1$, then there exists an $H'$ obtained from $H$ by a sequence of $2$-lifts without cycle of length at most $k$. We show that if the girth $g(H)=k$, then there exists a lift of $H$ with less number of $k$-cycles than $H$. Let $X$ be the random variable counting the number of $k$-cycles in a random $2$-lift of $H$. Every $k$-cycle of $H$ lifts to two $k$-cycles or a $2k$-cycle with probability $1/2$ each, so $\mathbb{E}X$ is exactly the number of $k$-cycles of $H$. But $H\cup H$ has two times as many $k$-cycles than $H$, so there must be a lift with strictly less number of $k$-cycles than $H$ has. Choose this $2$-lift and iterate this step to obtain an $H'$ with girth at least $k+1$.

\end{proof} 

\begin{Cor} (a) For any $d$--regular graph $G$, there exists a graph sequence $(G_i)_{i=0}^{\infty}$ such that $G_0=G$, $G_i$ is a $2$-lift of $G_{i-1}$ for $i\geq 1$, and $(G_i)$ is Benjamini--Schramm convergent to the $d$--regular infinite tree $\mathbb{T}_d$. 
\medskip

\noindent (b) For any $(a,b)$-biregular bipartite graph $G$, there exists a graph sequence $(G_i)_{i=0}^{\infty}$ such that $G_0=G$, $G_i$ is a $2$-lift of $G_{i-1}$ for $i\geq 1$, and $(G_i)$ is Benjamini--Schramm convergent to the $(a,b)$-biregular infinite tree $\mathbb{T}_{a,b}$.
\end{Cor}

\begin{proof}[Proof of Theorem~\ref{Schrijver} and \ref{Gurvits}] Let $0\leq p<1$. Choose a graph sequence $(G_i)_{i=0}^{\infty}$ such that $G_0=G$, $G_i$ is a $2$-lift of $G_{i-1}$ for $i\geq 1$, and $(G_i)$ is Benjamini--Schramm convergent to the $d$--regular infinite tree $\mathbb{T}_d$. Then by Corollary~\ref{mainlemmacor}
$$\lambda_{G_0}(p)\geq \lambda_{G_1}(p)\geq \lambda_{G_2}(p)\geq \dots $$
and 
$$\lim_{i\to \infty} \lambda_{G_i}(p)=\lambda_{\mathbb{T}_d}(p)$$
since $G_i$ converges to $\mathbb{T}_d$ (see Theorem~\ref{entropy}). Hence $\lambda_{G}(p)\geq \lambda_{\mathbb{T}_d}(p)$ for $0\leq p<1$.
Finally, for $p=1$ we have
$$\lambda_{G}(1)=\lim_{p\to 1} \lambda_{G}(p)\geq \lim_{p\to 1} \lambda_{\mathbb{T}_d}(p)=\lambda_{\mathbb{T}_d}(1).$$
Note that by part (e) of Proposition~\ref{asymp}, the inequality $\lambda_{G}(1)\geq \lambda_{\mathbb{T}_d}(1)$ is equivalent with
$$\frac{\ln \prm(G)}{v(G)}\geq \frac{1}{2}\ln \left( \frac{(d-1)^{d-1}}{d^{d-2}}\right)$$
which completes the proof of Theorem~\ref{Schrijver}.
\end{proof}

One can prove the following theorem the very same way.

\begin{Th} \label{entropy-biregular} For any $(a,b)$-biregular bipartite graph $G$ we have
$$\lambda_G(p)\geq \G_{a,b}(p)$$
for every $0\leq p\leq \min(\frac{a}{a+b},\frac{b}{a+b})$.
\end{Th}

With the same technique one can prove the following theorem.

\begin{Th} \label{integral-inequality} Let $G$ be a $d$--regular bipartite graph, and $t\geq 0$. Then
$$\int \frac{1}{2}\ln\left(1+tz^2\right)\, d\rho_G(z)\geq \int \frac{1}{2}\ln\left(1+tz^2\right)\, d\rho_{\mathbb{T}_d}(z).$$
\end{Th}

\begin{proof} Note that
$$\frac{\ln M(G,t)}{v(G)}=\int \frac{1}{2}\ln\left(1+tz^2\right)\, d\rho_G(z).$$
Let us choose a graph sequence $(G_i)_{i=0}^{\infty}$ such that $G_0=G$, $G_i$ is a $2$-lift of $G_{i-1}$ for $i\geq 1$, and $(G_i)$ is Benjamini--Schramm convergent to the $d$--regular infinite tree $\mathbb{T}_d$. By Lemma~\ref{mainlemma} we have
$$\frac{\ln M(G_0,t)}{v(G_0)}\geq \frac{\ln M(G_1,t)}{v(G_1)}\geq \frac{\ln M(G_2,t)}{v(G_2)}\geq \dots$$
and by the weak convergence of the measures $\rho_{G_i}$ (see Theorem~\ref{wc}) we have
$$\lim_{i\to \infty}\frac{\ln M(G_i,t)}{v(G_i)}=\lim_{i\to \infty}\int \frac{1}{2}\ln\left(1+tz^2\right)\, d\rho_{G_i}(z)=\int \frac{1}{2}\ln\left(1+tz^2\right)\, d\rho_{\mathbb{T}_d}(z).$$
Hence
$$\int \frac{1}{2}\ln\left(1+tz^2\right)\, d\rho_G(z)\geq \int \frac{1}{2}\ln\left(1+tz^2\right)\, d\rho_{\mathbb{T}_d}(z).$$
\end{proof}

Next we prove Theorem~\ref{direct} which is a direct consequence of the previous theorem.

\begin{proof}[Proof of Theorem~\ref{direct}] We can assume that $0\leq p<1$, for $p=1$ the claim follows from continuity.
We have seen that for $t\geq 0$
$$\frac{\ln M(G,t)}{v(G)}=\int \frac{1}{2}\ln\left(1+tz^2\right)\, d\rho_G \geq \int \frac{1}{2}\ln\left(1+tz^2\right)\, d\rho_{\mathbb{T}_d}.$$
Note that by Theorem~\ref{tree} we have
$$\frac{1}{2}\int \ln\left(1+tz^2\right)\, d\rho_{\mathbb{T}_d}=\frac{1}{2}\ln S_d(t),$$
where 
$$S_d(t)=\frac{1}{\eta_t^2}\left(\frac{d-1}{d-\eta_t}\right)^{d-2},$$
where 
$$\eta_t=\frac{\sqrt{1+4(d-1)t}-1}{2(d-1)t}.$$
Hence
$$M(G,t)\geq S_d(t)^n$$
for all $t\geq 0$. Now let
$$t=t(\mathbb{T}_{d},p)=\frac{p(d-p)}{d^2(1-p)^2}.$$
Then
$$\eta_t=\frac{1-p}{1-p/d},$$
and
$$S_d(t)=\frac{\left(1-\frac{p}{d}\right)^d}{(1-p)^2}.$$
Hence
$$M\left(G,\frac{p(d-p)}{d^2(1-p)^2}\right)\geq \frac{1}{(1-p)^{2n}}\left(1-\frac{p}{d}\right)^n.$$
After multiplying by $(1-p)^{2n}$, we get the claim of the theorem.
\end{proof}

We end this section with another corollary of Theorem~\ref{integral-inequality}. The so-called \emph{matching energy} was introduced by I. Gutman and S. Wagner \cite{G-W}, it is defined as follows:
$$ME(G)=\sum_{z_i: \mu(G,z_i)=0}|z_i|,$$
where all zeros are counted with its multiplicity. With our notation this is nothing else than
$$ME(G)=v(G)\int |z|\, d\rho_G(z).$$
The following theorem shows that if we normalize the matching energy by dividing by the number of vertices then among $d$--regular bipartite graphs its "minimum" is attained at the infinite $d$--regular tree $\mathbb{T}_d$.

\begin{Cor} Let $G$ be a $d$--regular bipartite graph. Then
$$\int |z|\, d\rho_G(z)\geq \int |z|\, d\rho_{\mathbb{T}_d}(z).$$
\end{Cor}

\begin{proof} Note that for any $z$ we have
$$|z|=\frac{1}{\pi}\int_0^{\infty} \frac{1}{t^2}\ln(1+t^2z^2)\, dt.$$
Hence
$$\int |z|\, d\rho_G=\int \left(\frac{1}{\pi}\int_0^{\infty} \frac{1}{t^2}\ln(1+t^2z^2)\, dt\right)\, d\rho_G(z)=$$
$$=\frac{1}{\pi}\int_0^{\infty}\frac{1}{t^2}\left(\int \ln(1+t^2z^2)\, d\rho_G(z) \right) \, dt\geq $$
$$\geq \frac{1}{\pi}\int_0^{\infty}\frac{1}{t^2}\left(\int \ln(1+t^2z^2)\, d\rho_{\mathbb{T}_d}(z) \right) \, dt=\int |z|\, d\rho_{\mathbb{T}_d}.$$
Since we integrated a non-negative function, it was allowed to interchange the integrals.
\end{proof}

\begin{Rem} Note that
$$\int |z|\, d\rho_{\mathbb{T}_d}(z)=\frac{d}{\pi}\left(2\sqrt{d-1}-(d-2)\arctan\left(\frac{2}{d-2}\sqrt{d-1}\right)\right).$$
\end{Rem}

\section{Proof of the Lower Matching Conjecture} \label{finish}

In this section we prove Theorem~\ref{main}. Here the main tool is that the matching polynomial has only real zeros, this gives sufficient information about its coefficients so that together with our results on the entropy function we can finish the proof of Theorem~\ref{main}. 
Let us mention that the argument in this section is more or less standard, a survey on related methods and results can be found in \cite{pit}.

\begin{proof}[Proof of Theorem~\ref{main}] We can assume that $0\leq p<1$ since for $p=1$, the statement reduces to Schrijver's theorem. Let $t$ be chosen such a way that $p(G,t)=p=\frac{k}{n}$. Then
$$m_k(G)=\frac{m_k(G)t^k}{M(G,t)}\exp (v(G)\lambda_G(p)).$$
Let
$$a_j=\frac{m_j(G)t^j}{M(G,t)}.$$
Then the probability distribution $(a_0,a_1,\dots ,a_n)$ has mean $\mu=k$. By the Heilmann--Lieb theorem, $\sum a_jx^j$ has only real zeros.
Then it is known that it is a distribution of the number of successes in independent trials. Indeed, let
$$M(G,t)=\prod_{i=1}^n(1+\gamma_it),$$
where $\gamma_i=\lambda_i^2$ with our previous notation, and
$$p_j=\frac{\gamma_jt}{1+\gamma_jt}.$$
If $I_j$ is the indicator variable that takes the value $1$ with probability $p_j$ and $0$ with probability $1-p_j$, then
$$\mathbb{P}(I_1+\dots +I_n=j)=a_j.$$
The advantage of this observation is that there is a powerful inequality for such distributions, namely Hoeffding's inequality.

\begin{Th}[Hoeffding's inequality \cite{Hoe}] Let $S$ be a random variable with probability distribution of the number of successes in $n$ independent trials. Assume that $\mathbb{E}S=np$. Let $b$ and $c$ integers satisfying $b\leq np\leq c$. Then
$$\mathbb{P}(b\leq X\leq c)\geq \sum_{j=b}^c{n \choose j}p^j(1-p)^{n-j}.$$
\end{Th}

In the particular case when $np=k$, we get that
$$a_k\geq {n \choose k}p^k(1-p)^{n-k}=p_{\mu}$$
with our previous notation. 

Putting everything together we obtain that
$$m_k(G)=\frac{m_k(G)t^k}{M(G,t)}\exp (v(G)\lambda_G(p))\geq p_{\mu}\exp (2n\G_d(p)).$$
In the last step we used that $\lambda_G(p)\geq \G_d(p)$ by Theorem~\ref{equivalent}.
\end{proof}

\begin{proof}[Proof of Theorem~\ref{biregular}]
The proof is completely analogous to the previous one. We have to use the inequality $\lambda_G(p)\geq \G_{a,b}(p)$, see Theorem~\ref{entropy-biregular}.
\end{proof}

\bigskip

\noindent \textbf{Acknowledgment.} The author is very grateful to the following people for the helpful conversations: Mikl\'os Ab\'ert, P\'eter Frenkel, Leonid Gurvits, Tam\'as Hubai, G\'abor Kun and Brendan McKay.

\end{document}